\documentclass[12pt,a4paper]{amsart}
\usepackage[margin=2.5cm]{geometry}
\usepackage{tikz,tkz-graph}
\usepackage{amsfonts,amsthm,amsmath}
\usepackage{amssymb}
\usepackage[utf8]{inputenc}
\usepackage{epic}
\usepackage{enumerate}
\usepackage{mathtools}
\usepackage{graphicx}
\usepackage{hyperref}
\usepackage{nicefrac}

\numberwithin{equation}{section}
\newtheorem{thm}{Theorem}[section]
\newtheorem{lemma}{Lemma}[section]

\theoremstyle{definition}
\newtheorem{definition}{Definition}[section]
\newtheorem{example}{Example}[section]

\theoremstyle{remark}
\newtheorem{remark}{Remark}[section]

\newtheorem*{remark*}{Remark}

\DeclareMathOperator{\Hom}{{\rm Hom}}

\newcommand{\Q}{\mathbb{Q}}
\newcommand{\Z}{\mathbb{Z}}
\newcommand{\N}{\mathbb{N}}
\newcommand{\Se}{\mathcal{S}}
\newcommand{\cS}{\mathcal{S}}
\newcommand{\cP}{\mathcal{P}}
\newcommand{\calv}{\mathcal{V}}
\newcommand{\al}{\alpha}
\newcommand{\w}{\omega}

\begin{document}

\title[Representability of SPM numerical semigroups]{Representability of systems of proportionally modular numerical semigroups}

\author[Zs. Baja]{Zsolt Baja}
\address{Babe\c{s}-Bolyai University, Faculty of Mathematics and Computer Science,\newline \hspace*{6mm} 
Str. Mihail Kog\u{a}lniceanu nr. 1, 400084 Cluj-Napoca, Romania}
\email{zsolt.baja@ubbcluj.ro}

\author[T. L\'aszl\'o]{Tam\'as L\'aszl\'o}
\address{Babe\c{s}-Bolyai University, Faculty of Mathematics and Computer Science,\newline \hspace*{6mm} 
Str. Mihail Kog\u{a}lniceanu nr. 1, 400084 Cluj-Napoca, Romania}
\email{tamas.laszlo@ubbcluj.ro}

\author[Zs. Nagy]{Zsuzsa Nagy}
\address{Babe\c{s}-Bolyai University, Faculty of Mathematics and Computer Science,\newline \hspace*{6mm} 
Str. Mihail Kog\u{a}lniceanu nr. 1, 400084 Cluj-Napoca, Romania}
\email{zsuzsa.nagy3@stud.ubbcluj.ro}


\thanks{Zs.B. was supported by the project `Singularities and Applications' - CF 132/31.07.2023 funded by the European Union - NextGenerationEU - through Romania's National Recovery and Resilience Plan.\\
T.L. is supported by the `J\'anos Bolyai Research Scholarship' of the Hungarian Academy of Sciences and partially supported by NKFIH Grant `\'Elvonal (Frontier)' KKP 144148. \\
Zs.N. was supported by  the `M\'arton \'Aron Scholarship' given by the Ministry of Foreign Affairs of Hungary, by the `Special Fellowship for Scientific Activities' provided by the Babe\c s-Bolyai University, and by the `KING CAROL I' Research Fellowship granted by the National Research Authority, no.  18RCI/18.12.2025. She was also partially supported by the Hungarian University Federation of Cluj and the BGA Zrt.}

\date{}

\begin{abstract}
In this short note we prove that every system of proportionally modular numerical semigroups is representable by a canonical equivariant resolution of a weighted homogeneous surface singularity with rational homology sphere link.  The construction starts from the quotient descriptions of proportionally modular numerical semigroups by two-generator numerical semigroups, realizes each quotient by a two-legged canonical equivariant resolution graph, and then glues these graphs with suitable multiplicities.  
\end{abstract}

\maketitle

\section{Introduction}

Numerical semigroups arising from the graded structure of the coordinate ring of normal weighted homogeneous surface singularities form an interesting and deeply active intersection between algebraic geometry, low-dimensional topology and combinatorics, cf. \cite{LN_sf,BL_flat,BLNgenus,BLN_PM}. When the link of the weighted homogeneous surface singularity is a rational homology sphere, this numerical semigroup is intricately governed by the topological Seifert structure. 

More precisely, one associates to a canonical equivariant resolution graph $\Gamma$, encoded by the Seifert invariants $(-b_0;(\al_i,\w_i)_{i=1}^d)$, a quasi-linear function 
\begin{equation}\label{defN}
N(\ell) = b_0 \ell - \sum_{i=1}^{d} \Big\lceil \frac{\ell\omega_i}{\alpha_i} \Big\rceil
\end{equation}
which counts the dimension of the homogeneous pieces of the graded ring, see section \ref{ss:reprsem}. Consequently, the corresponding numerical semigroup is completely determined by the Seifert invariants via the condition $N(\ell) \geq 0$. 
We call such a  numerical semigroup a \textit{representable semigroup}. That is they can be realized as the semigroup $\mathcal{S}_\Gamma$ associated with some canonical equivariant resolution graph $\Gamma$ of a normal weighted homogeneous surface singularity.

The study of representable numerical semigroups has received  attention recently, as it can be used to approach classical problems such as the Frobenius problem or developing formulas for the genus of numerical semigroups. The main problem of this topic is clearly the representability question, which seeks to answer whether all semigroups are representable.

On the other hand, this study has structural parallels with classes of semigroups arising purely from combinatorial and Diophantine contexts. An important class among these is the class of proportionally modular (PM) numerical semigroups, introduced by Rosales, Garc\'ia-S\'anchez, Garc\'ia-Garc\'ia, and Urbano-Blanco \cite{Garcia2002}. Defined as the set of non-negative integer solutions to a Diophantine inequality of the form $ax \pmod b \leq cx$, these semigroups appear natively in connection with the study of affine full semigroups \cite{delgado2008, RosalesUrbano2006}. A fundamental result by Rosales and Urbano-Blanco \cite{RosalesUrbano2006} establishes that every PM numerical semigroup can be represented as a quotient $\langle \alpha_1, \alpha_2 \rangle / n$ of a numerical semigroup generated by two relatively prime integers. 

Geometrically, it was shown in \cite{BLN_PM} that PM semigroups are perfectly compatible with two-legged canonical equivariant resolution graphs, meaning that they are representable by these Seifert structures given on cyclic quotient singularities.
\vspace{0.2cm}

A natural generalization of these semigroups is a system of proportionally modular (SPM) numerical semigroups, which arises as the solution set to a finite system of simultaneous proportionally modular inequalities \cite{Garcia2002}. Combinatorially, an SPM semigroup $\mathcal{S}$ is expressed as the intersection of finitely many PM semigroups $\cP_i$, $\mathcal{S} = \mathcal{P}_1 \cap \dots \cap \mathcal{P}_r$. Thanks to a structural refinement known as the Toms decomposition -- originally motivated by the classification of $\mathbb{C}^*$-algebras \cite{Toms2003} and adapted by the work of \cite{GarciaSanchezRosales2008} -- any SPM numerical semigroup can be rewritten using a unified denominator $n$ as 
$$
\mathcal{S} = \frac{1}{n} \bigcap_{i=1}^{r} \langle \alpha_{i1}, \alpha_{i2} \rangle
$$
While individual PM blocks are well-understood in the language of two-legged canonical equivariant resolution graphs, understanding the representability of their intersection introduces significant challenges. In general, it is not known that the intersection of representable semigroups is representable, as the interplay between the quasi-linear functions of different configurations can be highly non-trivial, see \cite{BL_flat}.

The aim of this paper is to bridge this gap in this case by proving the representability for any system of proportionally modular numerical semigroups. We use quotient representations of the individual PM blocks and take a sum of the associated graphs, c.f. section \ref{ss:sumgraph}. By carefully assigning multiplicities to each graph in the sum, we can control the behavior of the combined quasi-linear function. 
Using this idea, we will prove the main result of the article:

\begin{thm}\label{thm:main}
Every system of proportionally modular numerical semigroups is representable.
\end{thm}

The paper is organized as follows. In section \ref{s:WHS}, we review the foundational geometry and topology of weighted homogeneous surface singularities, their canonical equivariant resolution graphs, and the representable numerical semigroups. Section \ref{s:Ssf} outlines the definitions and properties of PM and SPM numerical semigroups, emphasizing the properties of the Toms decompositions. Section \ref{s:main} is dedicated to the proof of Theorem~\ref{thm:main}, where we explicitly construct the  sum of the quotient graphs with multiplicities and present concrete examples illustrating the result. Finally, we conclude with open questions regarding potential topological characterizations of these specific classes of star-shaped resolution graphs appearing in this context.

\section{Preliminaries: weighted homogeneous surface singularities and representable numerical semigroups}\label{s:DFP}

\subsection{On weighted homogeneous surface singularities}\label{s:WHS}
\subsubsection{\bf Definitions and notations}\label{ss:introWHS} A weighted homogeneous surface singularity $(X,0)$ is the germ at the origin of an affine surface $X$ with a good and effective  $\mathbb{C}^*$-action. In particular, its affine coordinate ring $R_X=\oplus_{\ell\geq 0} R_{X,\ell}$ is $\mathbb{Z}_{\geq 0}$-graded. We will assume further that $(X,0)$ is normal. 

There exists a modification $T\to X$ and a Seifert line bundle $T\to E_0$ with zero section $E_0$, where $E_0:=(X\setminus\{0\})/\mathbb{C}^*$ is a smooth compact curve and $T$ has at most a finite number of cyclic quotient singularities at the intersection of $E_0$ with each singular fiber. After resolving these singularities, we get a {\em canonical equivariant resolution} $\pi:\widetilde X\to X$ whose exceptional divisor $E:=\pi^{-1}(0)$ is a normal crossing divisor with a central irreducible component $E_0$, which can have self-intersection number $-1$.

The canonical equivariant resolution $\pi$ is a good resolution, however it is not necessarily the minimal good resolution. For example, in the case when $(X,0)$ is a cyclic quotient singularity, the $\mathbb{C}^*$-action/Seifert structure is not unique.  

Let $\Gamma$ be the dual graph of the resolution $\pi$, called the \emph{canonical equivariant resolution graph}. Then, by the above, $\Gamma$ is a `star-shaped' graph with central vertex $v_0$ and $d\ge 0$ legs (chain of vertices) connected to it.


The $\mathbb{C}^*$-action induces an $S^1$--Seifert action on the link $M$ of the singularity. In particular, $M$ is a negative definite Seifert 3-manifold characterized by its normalized Seifert invariants $Sf=(-b_0,g;(\alpha_i,\omega_i)_{i=1}^d).$
Each leg is determined by a pair of integers  $(\alpha_i,\omega_i)$,
where $0<\omega_i <\alpha_i$ and $\gcd(\alpha_i,\omega_i)=1$. The $i^{\mathrm{th}}$ leg has $\nu_i$ vertices, say $v_{i1},\ldots, v_{i\nu_i}$ ($v_{i1}$ is connected to the central vertex $v_0$)
with Euler decorations (self-intersection numbers)
$-b_{i1},\ldots, -b_{i\nu_i}$. This is encoded by the Hirzebruch--Jung (negative) continued fraction expansion
$$ \alpha_i/\omega_i=[b_{i1},\ldots, b_{i\nu_i}]=
b_{i1}-1/(b_{i2}-1/(\cdots -1/b_{i\nu_i})\dots) \ \  \ \ (b_{ij}\geq 2).$$
All these vertices (except $v_0$) have genus--decorations zero. The central vertex $v_0$ corresponds to the central genus $g$ curve $E_0$ with self-intersection number $-b_0$.
It will be also useful to define the integer $\omega_i'$ satisfying 
\begin{equation}\label{eq:w'}
\omega_i\omega_i'\equiv 1 \, (\mathrm{mod} \,\alpha_i), \ \ 0< \omega_i'<\alpha_i.
\end{equation}
One can prove that $\alpha_i=\det(\Gamma^i)$, the determinant of the $i^{th}$--leg $\Gamma^i$, $\omega_i=\det(\Gamma^i\setminus v_{i1})$ and $\omega_i'=\det(\Gamma^i\setminus v_{i\nu_i})$. 

For our purpose we always assume that the link $M$ is a  (Seifert) rational homology sphere, or equivalently, the curve $E_0$ has $g=0$. Therefore, as a simplification, for the Seifert invariants we will only write $Sf=(-b_0;(\alpha_i,\omega_i)_{i=1}^d)$. 

For more details on this section we refer to \cite{OW,JanNeu}.

\subsubsection{\bf Lattices, cycles and numerical invariants}\label{ss:numinv}
In general, the resolution space $\widetilde X$ can be regarded as the plumbed $4$-manifold associated with $\Gamma$ with boundary $M$. 
We define the lattice $L$ as $H_2(\tilde X,\Z)$, endowed with the non-degenerate negative definite intersection form $(\, , \,)$. If $\{E_v\}_{v\in }$ are irreducible exceptional divisors, then $L$ is freely generated by the (classes of) $E_v$, that is, $L=\oplus_{v\in \calv} \Z\langle E_v \rangle$. The elements of $L$ are called (integral) cycles. 

The dual lattice $L':= \Hom(H_2(\widetilde X,\Z),\Z)$ can be identified with $H_2(\widetilde{X}, M, \Z)$. Moreover, the quotient is a finite group $L'/L\cong H_1(M,\Z)$, which will be denoted by $H$. Since the intersection form is non--degenerate, $L'$ embeds into the lattice of rational cycles $ L_{{\mathbb Q}}:=L\otimes {\mathbb Q}$, and it can be identified with the rational cycles satisfying $\{l'\in L_{{\mathbb Q}}\,:\, (l',L)_{{\mathbb Q}}\subset \Z\}$, where $(\,,\,)_{{\mathbb Q}}$ denotes the extension of the intersection form to $L_{{\mathbb Q}}$.
 Hence, we can regard $L'$ as $\oplus_{v\in \calv} \Z\langle E^*_v \rangle$, the lattice generated by the (anti-)dual cycles $E^*_v\in L_{{\Q}}$, $v\in \calv$, where $(E_u^*,E_v)_{{\Q}}=-\delta_{u,v}$ (Kronecker delta) for any $u,v\in \calv$.

We define the anti-canonical cycle $Z_K\in L'$ as the unique solution of the adjunction equations $(Z_K,E_v)=-b_v+2$ for all $v$. We can also express $Z_K$ as  
 \begin{equation}\label{ZK}
 Z_K-E=\sum_{v\in\mathcal{V}}(\delta_v-2)E^*_v,
 \end{equation}
where we denote $E:=\sum_{v\in\mathcal{V}}E_v$ and $\delta_v$ is the valency of the vertex $v$.
\vspace{0.2cm}

Let us assume that $M$ is a negative definite Seifert rational homology sphere as in the previous section \ref{ss:introWHS}. One defines the orbifold Euler number of  $M$ as $e:=-b_0+\sum_i\omega_i/\alpha_i$. Then the negative definiteness of the intersection form is equivalent to $e<0$. Furthermore, using the notation  $\alpha:=\mathrm{lcm}(\alpha_1,\ldots,\alpha_d)$, one can express the order of the group $H$ and the class of $[E^*_0]\in H$ as follows:
\begin{equation}\label{eq:sei2}
 |H|=\alpha_1\cdots\alpha_d|e| \ \ \mbox{and} \ \ \textrm{ord}([E^*_0])=\alpha|e|.
\end{equation}
In particular, if $M$ is an integral homology sphere then necessarily all $\alpha_i$'s are pairwise relatively prime and (\ref{eq:sei2}) gives the Diophantine equation $(b_0-\sum_i\omega_i/\alpha_i)\alpha=1$ which uniquely determines all $\omega_i$ and $b_0$. 

The $E_0$-coefficient of the dual cycles
 $E^*_v$ associated with the end-vertices (i.e. vertices with $\delta_v=1$) can be computed by $(E_v^*,-E^*_0)=1/(|e|\alpha_v)$, while the $E_0$-coefficient of
 $E^*_0$ is  $(E_0^*,- E^*_0)=1/|e|$.
These, applied to (\ref{ZK}), express the $E_0$-coefficient of $Z_K$ as $\gamma+1$, where
\begin{equation}\label{eq:gamma}
\gamma:=\frac{1}{|e|}\cdot \Big( d-2-\sum_{i=1}^d \frac{1}{\alpha_i}\Big)\in \mathbb{Q}.
\end{equation}
Note that this number $\gamma$  has a central importance in the study of properties of
weighted homogeneous surface singularities or Seifert rational homology spheres, see eg. \cite{nembook,neumann.abel}. 
One can prove  that $\gamma|e|=d-2-\sum_i 1/\alpha_i$ is negative if and only if  $d\leq 2$, or if  $d=3$ and $\sum_i1/\alpha_i>1$ (or, topologically, $\pi_1(M)$ should be  finite). In these cases $(X,0)$ is a quotient singularity, hence rational. 

All the details of this section can be found in \cite{nembook}. See also \cite{neumann.abel,Nfive,NOSZ}.

\subsection{Representable numerical semigroups \cite{LN_sf,BL_flat}}\label{ss:reprsem}
\subsubsection{\bf Definitions}
Let $(X,0)$ be a normal weighted homogeneous surface singularity.  The set $\cS:=\{\ell\in\mathbb{Z}_{\geq 0} | R_{X,\ell}\neq 0\}$ is a numerical semigroup by the grading property. 

Let us consider a canonical equivariant resolution of $(X,0)$ with dual graph $\Gamma$. By \cite{Pinkham1977}, we know that 
the complex structure of $(X,0)$ is completely recovered by the Seifert invariants and the configuration of points $\{P_i:=E_0\cap E_{i1}\}_{i=1}^d \subset E_0$, where $E_{i1}$ is the irreducible component corresponding to $v_{i1}$ in $\Gamma$. Furthermore, there is a {\em Dolgachev--Pinkham--Demazure} formula which expresses the graded ring, using a special divisor  $D^{(\ell)}:=\ell (-E_0|_{E_0})-\sum_{i=1}^d\lceil\ell\omega_i/\alpha_i\rceil P_i$, as 
\begin{equation}\label{eq:DPD}
R_X=\oplus_{\ell\geq 0} R_{X,\ell}=\oplus_{\ell\geq 0}
H^0(E_0,\mathcal{O}_{E_0}(D^{(\ell)})),
\end{equation}
where $\lceil x\rceil$ denotes the smallest integer greater than or equal to $x$.

In our case, when $M$ is a rational homology sphere,  (\ref{eq:DPD}) implies that $\dim(R_{X,\ell})=\max\{0,1+N(\ell)\}$ is topological, where $N(\ell)$
is  the quasi-linear function defined by 
\begin{equation}\label{defN}
N(\ell):=\deg D^{(\ell)}=b_0\ell-\sum_{i=1}^d\Big\lceil \frac{\ell \omega_i}{\alpha_i}\Big\rceil,
\end{equation}
and the semigroup $\cS$ can be described purely with the Seifert invariants as 
\begin{equation}\label{eq:sgptop}
\mathcal{S}_{(X,0)}=\{\ell\in\mathbb{Z} \ | \ N(\ell)\geq 0\}.
\end{equation}

Note that $-\lceil x\rceil \leq -x$ implies $N(\ell)\leq |e|\ell$,
hence  $N(\ell)<0$ for $\ell<0$. 

Note that in this case $\cS$ is associated with the canonical equivariant resolution graph $\Gamma$ (or, equivalently, with the Seifert structure), hence the notation $\cS_{\Gamma}$ will be used. 

A numerical semigroup $\cS$ that can be realized as $\cS_\Gamma$ for some canonical equivariant resolution graph $\Gamma$ is called {\em representable numerical semigroup}. Accordingly, we say that the $\Gamma$ (or the Seifert structure) is a {\it representation} of the numerical semigroup $\cS$.










\subsubsection{\bf Quotient semigroups and the sum of  graphs}\label{ss:sumgraph}

Associated with a numerical semigroup $\cS$ and a positive integer $n\in\mathbb{N}_{>0}$, one defines the quotient semigroup 
\[
   \cS/n:=\{\ell\in\mathbb{N}:n\ell\in\cS\}.
\]
If $\cS=\cS_\Gamma$ is representable and $N(\ell)$ is the quasi-linear function associated with $\Gamma$, then $\cS/n$ can be represented by a new graph $\Gamma^{(n)}$ whose associated quasi-linear function is 
$$N^{(n)}(\ell):=N(n\ell)$$.

Now, let $\Gamma_1$ and $\Gamma_2$ be two canonical equivariant resolution graphs with Seifert invariants
\[
    Sf_1=\left(-b_0; (\al_i,\w_i )_{i=1}^{m_1} \right)
    \quad\text{and}\quad
    Sf_2=\left(-c_0; (\beta_j,w_j )_{j=1}^{m_2} \right).
\]
We define the sum $\Gamma:=\Gamma_1+\Gamma_2$ as the  graph determined by
\[
    Sf=\left(-b_0-c_0; (\al_i,\w_i)_{i=1}^{m_1},(\beta_j,w_j )_{j=1}^{m_2}  \right).
\]
If $e_1,e_2$ and $e$ are the corresponding orbifold Euler numbers, then $e=e_1+e_2<0$, so the construction is well defined.  Moreover, if $N_1,N_2$ and $N$ are the quasi-linear functions associated with $\Gamma_1,\Gamma_2$ and $\Gamma$, respectively, then one gets
\[
   N(\ell)=N_1(\ell)+N_2(\ell).
\]
The following result was proved in \cite{BL_flat}: 
\begin{lemma}
    \label{thm:sumbound}
    If $\Gamma_1$ and $\Gamma_2$ are canonical equivariant resolution graphs then
    \begin{equation}
        \label{eq:sumbound}
        \Se_{\Gamma_1}\cap\Se_{\Gamma_2}\subset 
        \Se_{\Gamma_1+\Gamma_2}
        \subset\Se_{\Gamma_1}+\Se_{\Gamma_2}.        
    \end{equation} 
\end{lemma}

The concepts presented in this section were defined in \cite{BL_flat} and will be important for us in this article.

\section{Proportionally modular and system of proportionally modular numerical semigroups}\label{s:Ssf}

\subsection{Proportionally modular numerical semigroups (PM) and their geometry}\label{s:PM}

Proportionally modular numerical semigroups were introduced by Rosales,
Garc\'ia-S\'anchez, Garc\'ia-Garc\'ia and Urbano-Blanco in
\cite{Garcia2002}.  

In the sequel we recall only those properties and facts that will be used in our 
construction of representations for SPM numerical semigroups. Further details on this class, regarding their equivalent Diophantine descriptions, 
decomposition results, or about the geometry, can be found in
\cite{Garcia2002,RoblesPerezRosales2008,RosalesUrbano2006, BLN_PM} and in the references therein.

\begin{definition}\label{def:pm}
Let $c<a<b$ be positive integers.  We set 
$
S(a,b,c):=\{x\in \Z_{\geq 0}\ :\ ax \bmod b\leq cx\}$.
If a numerical semigroup $\cP$ satisfies $\cP=S(a,b,c)$ for some $a,b,c$, then $\cP$ is called
a \emph{proportionally modular numerical semigroup} (PM semigroup for short).  In this case 
$S(a,b,c)$ is a \emph{Diophantine representation} of $\cP$.
The integers $a,b,c$ are called the factor, modulus and proportion,
respectively.  In the special case $c=1$, the semigroup $S(a,b,1)$ is called a
modular numerical semigroup.
\end{definition}

Consider the quotient semigroups
$
\langle \alpha_1,\alpha_2\rangle/n
=\{\ell\in\N\ :\ n\ell\in \langle \alpha_1,\alpha_2\rangle\}
$,
where $\langle \alpha_1,\alpha_2\rangle$ denotes the numerical semigroup
generated by $\alpha_1$ and $\alpha_2$, $\gcd(\al_1,\al_2)=1$.  A fundamental result of
Rosales and Urbano-Blanco \cite{RosalesUrbano2006} says that every PM numerical
semigroup $\cP$ is of this form:
\[
\cP=\frac{\langle \alpha_1,\alpha_2\rangle}{n}
\quad \mbox{for some } \ \al_1,\al_2 \ \ \mbox{and} \ \ n \ \ \mbox{with} \ \gcd(\alpha_1,\alpha_2)=1,
\]
and, conversely, every such quotient is proportionally modular.  

The minimal generators of a PM semigroup have a characteristic `bamboo' behaviour, c.f.  \cite{Garcia2002,BLN_PM}.  More
precisely, the minimal generators can be ordered, say as
$g_1,\ldots,g_e$, so that consecutive generators are relatively prime and satisfy
\[
g_{i-1}+g_{i+1}\equiv 0 \pmod {g_i},
\qquad 2\leq i\leq e-1,
\]

This bamboo structure can be extended to non-minimal sets of generators as well, and using toric geometry one can construct associated with every such bamboo set of generators a two-legged canonical equivariant resolution graph, which represents the given proportionally modular numerical semigroup; see \cite{BLN_PM}.

Finally, we will prove the following observation.
\begin{lemma}\label{lem:N=-1}
Let $\Gamma$ be a two-legged representation of a  proportionally modular numerical semigroup $\cP$.
If $\ell\in\N$ such that $N_\Gamma(\ell)<0$, then $N_\Gamma(\ell)=-1$.
\end{lemma}

\begin{proof}
In section \ref{ss:numinv} we have discussed already that if a $\Gamma$ has $d=2$ legs, then $\gamma<0$. On the other hand, \cite[Prop. 3.2.11(b)]{LN_sf} one deduces that $N(\ell)\geq -1$ whenever $\ell>\gamma$. Thus $N_\Gamma(\ell)\geq -1$ for all $\ell\in\N$ in this case.  As $N_\Gamma$ is integer-valued, the condition $N_\Gamma(\ell)<0$ forces $N_\Gamma(\ell)=-1$.
\end{proof}

\subsection{System of proportionally modular numerical semigroups (SPM)}

\subsubsection{}
The concept of the system of proportionally modular numerical semigroups appeared at the same time as PM numerical semigroups in \cite{Garcia2002}. 

\begin{definition}
Let $a_i, b_i, c_i$, $i\in\{1,\dots,r\}$ be positive integers and consider the system of inequalities
\[
\begin{cases}
a_1 x \bmod b_1 \leq c_1 x, \\
\vdots \\
a_r x \bmod b_r \leq c_r x.
\end{cases}
\]

Then the set of all non-negative integer solutions $\cS$ is called a \emph{system of proportionally modular numerical semigroups}. We say further that $\cS$ is represented by a \emph{system of Diophantine representations} $\{S(a_i,b_i,c_i)\}_i$.

Let $\cP_i:=S(a_i,b_i,c_i)$ be the PM numerical given by the corresponding Diophantine representation. Then $\cS=\cP_1\cap\dots\cap\cP_r$ provides a \emph{PM decomposition} of the numerical semigroup $\cS$. 

For simplicity, we will refer to the system of proportionally modular numerical semigroups as SPM numerical semigroups.
\end{definition}

In  \cite{Garcia2002}, the authors provide also an algorithm which decides whether a numerical semigroup is an SPM or not. As a consequence of the quotient structure of the PM semigroups, every SPM can be written as 
$$\cS=\langle \al_{11},\al_{12}\rangle/n_1\cap\dots\cap\langle \al_{r1},\al_{r2}\rangle/n_r.$$
This property was further refined in the context of Toms decomposition in \cite{GarciaSanchezRosales2008}. In the sequel, we present this concept in more detail.

\subsubsection{\bf Toms decomposition of an SPM}[\cite{Toms2003,GarciaSanchezRosales2008}]\label{ss:Toms}

Let $\cS$ be a numerical semigroup. We say that $\cS$ has a \emph{Toms decomposition} if there exist positive integers $\al_{i1}, \al_{i2}$ ($i\in\{1,\dots,r\}$) and a unique $n\in \Z_{>0}$ such that
\begin{itemize}
\item[(1)] $\gcd(\al_{i1},\al_{i2})=\gcd(n,\al_{ij})=1$ for any $i \in \{1,\dots,r\}$, $j\in\{1,2\}$ and
\item[(2)] $\cS = \frac{1}{n}\bigcap_{i=1}^r \langle \al_{i1}, \al_{i2} \rangle.$
\end{itemize}

First of all, note that one has $\frac{1}{n}\bigcap_{i=1}^r \langle \al_{i1}, \al_{i2} \rangle=\bigcap_{i=1}^r\langle \al_{i1}, \al_{i2} \rangle/n$. Hence, by the assumption (1), every Toms block $\cP_i:=\langle \al_{i1}, \al_{i2} \rangle/n$  is a PM semigroup. In particular,  every numerical semigroup having a Toms decomposition is an SPM. Furthermore, by Rosales and Garc\'ia-S\'anchez one knows that the converse is also true. 

\begin{thm}[\cite{GarciaSanchezRosales2008}]
Every SPM numerical semigroup admits a Toms decomposition.
\end{thm}

\begin{example}
\label{ex:toms}
As a short example, let us consider the semigroup $\cS = \langle11, 14, 15, 18, 19, 21\rangle$. 
Using the algorithm of \cite{Garcia2002}, one can check that $\cS$ is an SPM and  list its decompositions. For example, we can find the following decompositions consisting of two blocks:  
\begin{align*}
\cS
= \langle 5,8,11,14, 17 \rangle
   \cap \langle 7, 11, 15, 19 \rangle
&= \langle 3, 11, 19 \rangle \cap \langle 7, 11, 15, 19 \rangle \\
&= \langle 5, 8, 11, 14 \rangle \cap \langle 7, 11, 15, 19 \rangle.
\end{align*}
Using quotients of $2$ generators semigroups, these identities can be read as
\begin{equation*}
\cS= \frac{\langle 5, 17 \rangle}{4}
   \cap \frac{\langle 7, 19 \rangle}{3}
= \frac{\langle 3, 19 \rangle}{2} \cap \frac{\langle 7, 19 \rangle}{3} = \frac{\langle 5, 14 \rangle}{3} \cap \frac{\langle 7, 19 \rangle}{3}.
\end{equation*}
As we can see, only the last decomposition is a Toms decomposition, whereas in the first two cases  the denominators do not coincide.
\end{example}

\begin{remark}
We note that the Toms decomposition of an SPM numerical semigroup is not unique, consider eg. the simple example $\langle 4,6,7,9\rangle=\langle 2,7\rangle\cap\langle 3,4\rangle=\langle 2,5\rangle\cap \langle 3,4\rangle$.
\end{remark}

\section{Representations for SPM numerical semigroups}\label{s:main}

In this section we prove Theorem \ref{thm:main} by constructing a concrete representation of an SPM numerical semigroup, based on its Toms decomposition.

\subsection{The quotient representation of PM semigroups}\label{s:PMquotrep} First of all let $\cP$ be a PM numerical semigroup. We fix a quotient structure $\cP=\langle \al_1, \al_2\rangle/n$ as in section \ref{s:PM}.  

Since $\cP$ is representable, in the sequel we will choose one of its representations, $\Gamma$, which encodes the properties of its fixed quotient structure and it will be convenient for our purpose. 

Indeed, by \cite[Example 3.2]{BL_flat} we know that  $\langle  \al_1, \al_2\rangle$ is representable by a graph $\Gamma_{\langle  \al_1, \al_2\rangle}$ with two legs, defined by the Seifert invariants:
\[
Sf_{\Gamma_{\langle  \al_1, \al_2\rangle}}=(-1, (\al_1, \beta_{1}), (\al_2, \beta_{2})),
\]
where the equation $\al_1\al_2 - \beta_{1} \al_2 -\beta_{2}\al_1=1$ together with $0\leq \beta_j<\alpha_j$ determines uniquely the integers $\beta_j$, $j=1,2$. Then, for the fixed $n\geq 1$ one can construct from $\Gamma_{\langle  \al_1, \al_2\rangle}$ a representation $\Gamma$ of $\cP$ as follows.

The quasi-linear function associated with $\Gamma_{\langle  \al_1, \al_2\rangle}$ is $N(\ell)=\ell-\lceil \beta_{1}\cdot \ell/\al_1\rceil-\lceil \beta_{2}\cdot \ell/\al_2\rceil$. That is $\langle \al_1, \al_2\rangle=\{\ell\in \Z \ : \ N(\ell)\geq 0\}$. Then the quasi-linear function $N(o\ell)$ defines the quotient semigroup $\cP$ and provides the Seifert invariants for one of its representation $\Gamma$, see \cite{BL_flat}[Remark 3.3.2]. 

An advantage of this representation is that the determinant of $\Gamma$ is exactly the denominator $n$, as we will show in the next lemma.

\begin{lemma}
\label{lem:1}
Let $\cP$ be a PM numerical semigroup written as a quotient  $\langle \al_1, \al_2\rangle/n$ with $\al_1,\al_2,n$ pairwise relatively prime positive integers, and consider the representation $\Gamma$ of $\cP$ as above.  Then $\det(\Gamma)=n$.
\end{lemma}

\begin{proof}

Recall from \cite{BL_flat} that we can write 
\begin{equation}\label{eq:Nol}
N(n\ell) = n\ell - \Big\lceil \frac{n\ell\cdot\beta_{1}}{\al_1} \Big\rceil - \Big\lceil \frac{n\ell\cdot \beta_{2}}{\al_2} \Big\rceil=\Big(n\cdot|e_{\Gamma_{\langle  \al_1, \al_2\rangle}}|+\frac{\omega_{1}}{\al_1}+\frac{\omega_{2}}{\al_2}\Big)\ell-\Big\lceil \frac{\ell \omega_{1}}{\al_1} \Big\rceil-\Big\lceil \frac{\ell \omega_{2}}{\al_2} \Big\rceil,
\end{equation}
where $0\leq \omega_{j}<\al_j$ satisfy $n\beta_{j} \equiv \omega_{j} \pmod{\al_j}$ for every $j=1,2$. Since $n|e_{\Gamma_{\langle  \al_1, \al_2\rangle}}|+\sum_{j=1,2}\omega_{j}/\al_j\in \Z_{>0}$, the expression (\ref{eq:Nol}) gives us the Seifert invariants 
\begin{equation}\label{eq:SfPM}
Sf_{\Gamma}=(-(n\cdot|e_{\Gamma_{\langle  \al_1, \al_2\rangle}}|+\sum_{j=1,2}\omega_{j}/\al_j);(\al_1,\omega_{1}),(\al_2,\omega_{2}))
\end{equation}
of $\Gamma$. In particular, $e_{\Gamma}=ne_{\Gamma_{\langle  \al_1, \al_2\rangle}}$ and 
$$\det(\Gamma)=\al_1\al_2|e_{\Gamma}|=\al_1\al_2n|e_{\Gamma_{\langle  \al_1, \al_2\rangle}}|=n,$$
since $|e_{\Gamma_{\langle  \al_1, \al_2\rangle}}|=1/(\al_1\al_2)$ is implied by the fact that $\Gamma_{\langle  \al_1, \al_2\rangle}$ is a canonical equivariant resolution graph of the smooth germ $(\mathbb{C}^2,0)$, hence $\det(\Gamma_{\langle  \al_1, \al_2\rangle})=1$.

\end{proof}

\begin{remark}\label{rem:trivial}
If we have $n\cdot|e_{\Gamma_{\langle  \al_1, \al_2\rangle}}|+\sum_{j=1,2}\omega_{j}/\al_j>1$ for the decoration of the central vertex in $\Gamma$, then $\cP=\cS_{\Gamma}=\N$. 
\end{remark}

\subsection{The sum of quotient representations with multiplicity}\label{s:gamma}
Let $\cS$ be an SPM numerical semigroup and we choose a  decomposition $\cS=\cP_1\cap\dots\cap\cP_r$ such that the PM semigroups are written as quotients $\cP_i=\langle \al_{i1},\al_{i2}\rangle/n_i$ for any $i\in\{1,\dots,r\}$, $\al_{i1},\al_{i2},n_i$ satisfying the properties of section \ref{s:PM}.

Associated with every $\cP_i$, $i=1,\dots,r$, we consider the representation $\Gamma_i$ constructed in section \ref{s:PMquotrep} and its quasi-linear function $N_{\Gamma_i}(\ell)$. We may assume that the Seifert invariants of $\Gamma_i$ are $Sf_{\Gamma_i}=(-1;(\alpha_{ij}, \omega_{ij})_{j=1,2})$, otherwise by Remark \ref{rem:trivial} $\cP_i$ is trivial and can be omitted from the decomposition. Let us denote by $e_i$ the orbifold euler number associated with $\Gamma_i$ and assume that  
$$|e_1|\leq |e_2| \leq \dots \leq |e_r|.$$
Now, for any $i,j\in\{1,\dots,r\}$ with $i<j$ define the positive integer
$$k_{ij} := \Big\lceil\frac{e_j}{e_i} \Big\rceil-1.$$ 
Furthermore, we associate recursively with any $\Gamma_i$ the following multiplicities: 
\begin{equation}\label{eq:mult}
q_r = 1 \quad \mbox{and} \quad q_i = \sum_{i<j\leq r} k_{ij} q_j + 1 \quad \text{for all } \ i \in \{1, \dots, r-1\}.
\end{equation}

Now, using section \ref{ss:sumgraph}, we can define a new graph by
\begin{equation}\label{SPMrep}
\Gamma := \sum_{i=1}^{r} q_i \Gamma_i.
\end{equation}
We will now prove that $\Gamma$ is a representation of the SPM numerical semigroup $\cS$.

\begin{thm}\label{thm:2}
$\cS_\Gamma=\cS$, that is the semigroup associated with $\Gamma$ is exactly the SPM semigroup $\cS$.
\end{thm}

\begin{proof}
We consider the quasi-linear function $N_\Gamma(\ell)$ associated with $\Gamma$, that is $\cS_{\Gamma}=\{\ell \in \N \ : \ N_\Gamma(\ell) \geq 0\}$. Since $\cS=\cap_{i=1}^r\ \cP_i$, the identification $\cS_\Gamma=\cS$ is equivalent with  
\begin{center}
$N_\Gamma(\ell) \geq 0$ if and only if $N_{\Gamma_i}(\ell) \geq 0$ for all $i = 1,\dots, r$ and $\ell \in \mathbb{N}$.
\end{center}
Note that $\cS\subset \cS_{\Gamma}$ is immediate since  we have $N_{\Gamma}=\sum_i q_iN_{\Gamma_i}$, see \ref{ss:sumgraph}. 

Let us now prove $\cS_{\Gamma}\subset \cS$. 

Assume that there exists an $\ell$ such that $N_\Gamma(\ell) \geq 0$ and $N_{\Gamma_j}(\ell) < 0$ for some $j\in\{1,\dots,r\}$, and we fix such an $\ell$ and index $j$. 

By Corollary \ref{lem:N=-1}, for this $\ell$ and $j$ we must have $N_{\Gamma_j}(\ell)=-1$. Furthermore, we will show that   
\begin{equation}\label{eq:ejceil}
    0 < \ell|e_j| < 1.
\end{equation}

Indeed, let us assume that  $\ell e_j\leq -1$. Then,  combining this inequality with $N_{\Gamma_j}(\ell)=b_j\ell-\sum_{i=1,2}\lceil\ell\omega_{ji}/\al_{ji}\rceil=-1$ would imply that  
\[
\sum_{i=1}^{2} \frac{\ell \,\omega_{ji}}{\alpha_{ji}} - \left\lceil \frac{\ell \,\omega_{ji}}{\alpha_{ji}} \right\rceil \leq -2,
\]
which is a contradiction. Hence, we have (\ref{eq:ejceil}).

Next, for any index $t\neq j$ we consider the following inequality:

\begin{align}\label{eq:Nt}
N_{\Gamma_t}(\ell)=\ell 
  - \left\lceil \frac{\omega_{t1}\,\ell}{\alpha_{t1}} \right\rceil
  - \left\lceil \frac{\omega_{t2}\,\ell}{\alpha_{t2}} \right\rceil \leq  \ell |e_t| .
\end{align}

Now, if $t < j$, the assumption $|e_t|\leq|e_j|$, (\ref{eq:ejceil}) and (\ref{eq:Nt}) imply  $N_{\Gamma_t} \leq 0$.

If we consider an index $t$ with $t > j$,  
then also with (\ref{eq:ejceil}) and (\ref{eq:Nt}) we deduce that  
\begin{equation}\label{eq:tgeqj}
N_{\Gamma_t}(\ell)\leq  \frac{|e_t|}{|e_j|}\ell |e_j| <  \frac{|e_t|}{|e_j|},
\end{equation}
hence we get $N_{\Gamma_t}(\ell)\leq\lceil|e_t|/|e_j|\rceil-1=k_{jt}$.

Combining these cases we can bound $N_{\Gamma}(\ell)$ as follows:
\begin{align*}
N_{\Gamma}(\ell)= \sum_{i=1}^{j-1} q_i N_{\Gamma_i}(\ell)
- q_j
+ \sum_{i=j+1}^{r} q_i N_{\Gamma_i}(\ell)\leq
- q_j
+ \sum_{i=j+1}^{r} q_i k_{ji} = -1.
\end{align*}
which contradicts our original assumption. 

Therefore we must have $N_{\Gamma_j} \geq 0$ for all $j = 1,\dots,r$, which concludes the proof.
\end{proof}

\begin{example}
We illustrate our construction by the following example. Consider the SPM semigroup $\cS=\langle 11,14,15,18,19,21\rangle$ and we look at its decomposition 
$\cS=\cP_1\cap \cP_2$, where $\cP_1=\langle   7, 11, 15, 19 \rangle=\langle 7,19\rangle/3$ and $\cP_2=\langle  5, 8, 11, 14, 17 \rangle=\langle 5,17\rangle/4$, see  Example \ref{ex:toms}. 

By \ref{s:PMquotrep} we can construct the quotient representations $\Gamma_i$ associated with $\cP_i$: $\Gamma_1$ is defined by the Seifert invariants $Sf_{\Gamma_1}=(-1;(7,5),(19,5))$, while $\Gamma_2$ is defined by $Sf_{\Gamma_2}=(-1;(5,3),(17,6))$. Then, one calculates $e_1=-3/133$, $e_2=-4/85$ and $k_{12}=\lceil e_2/e_1\rceil-1=2$. So the multiplicities of $\Gamma_i$ will be
$$q_1=k_{12}+1=3 \quad \mbox{and} \quad q_2=1$$ 
which, by (\ref{SPMrep}), provides the graph  $\Gamma=3\Gamma_1+\Gamma_2$ of $\cS$, shown by Figure \ref{fig:ex1}.

\begin{figure}[h!]
\centering

\tikzset{every picture/.style={line width=0.75pt}} 

\begin{tikzpicture}[x=0.75pt,y=0.75pt,yscale=-1,xscale=1]

\draw    (193,109) -- (251,189) ;
\draw [shift={(251,189)}, rotate = 54.06] [color={rgb, 255:red, 0; green, 0; blue, 0 }  ][fill={rgb, 255:red, 0; green, 0; blue, 0 }  ][line width=0.75]      (0, 0) circle [x radius= 3.35, y radius= 3.35]   ;
\draw [shift={(222,149)}, rotate = 54.06] [color={rgb, 255:red, 0; green, 0; blue, 0 }  ][fill={rgb, 255:red, 0; green, 0; blue, 0 }  ][line width=0.75]      (0, 0) circle [x radius= 3.35, y radius= 3.35]   ;
\draw [shift={(193,109)}, rotate = 54.06] [color={rgb, 255:red, 0; green, 0; blue, 0 }  ][fill={rgb, 255:red, 0; green, 0; blue, 0 }  ][line width=0.75]      (0, 0) circle [x radius= 3.35, y radius= 3.35]   ;
\draw    (251,189) -- (276,226) ;
\draw [shift={(276,226)}, rotate = 55.95] [color={rgb, 255:red, 0; green, 0; blue, 0 }  ][fill={rgb, 255:red, 0; green, 0; blue, 0 }  ][line width=0.75]      (0, 0) circle [x radius= 3.35, y radius= 3.35]   ;
\draw    (222,149) -- (314,150) ;
\draw [shift={(314,150)}, rotate = 0.62] [color={rgb, 255:red, 0; green, 0; blue, 0 }  ][fill={rgb, 255:red, 0; green, 0; blue, 0 }  ][line width=0.75]      (0, 0) circle [x radius= 3.35, y radius= 3.35]   ;
\draw [shift={(268,149.5)}, rotate = 0.62] [color={rgb, 255:red, 0; green, 0; blue, 0 }  ][fill={rgb, 255:red, 0; green, 0; blue, 0 }  ][line width=0.75]      (0, 0) circle [x radius= 3.35, y radius= 3.35]   ;
\draw    (222,149) -- (168,226) ;
\draw [shift={(168,226)}, rotate = 125.04] [color={rgb, 255:red, 0; green, 0; blue, 0 }  ][fill={rgb, 255:red, 0; green, 0; blue, 0 }  ][line width=0.75]      (0, 0) circle [x radius= 3.35, y radius= 3.35]   ;
\draw [shift={(195,187.5)}, rotate = 125.04] [color={rgb, 255:red, 0; green, 0; blue, 0 }  ][fill={rgb, 255:red, 0; green, 0; blue, 0 }  ][line width=0.75]      (0, 0) circle [x radius= 3.35, y radius= 3.35]   ;
\draw [shift={(222,149)}, rotate = 125.04] [color={rgb, 255:red, 0; green, 0; blue, 0 }  ][fill={rgb, 255:red, 0; green, 0; blue, 0 }  ][line width=0.75]      (0, 0) circle [x radius= 3.35, y radius= 3.35]   ;
\draw    (222,149) -- (251,104) ;
\draw    (251,104) -- (300,29) ;
\draw [shift={(300,29)}, rotate = 303.16] [color={rgb, 255:red, 0; green, 0; blue, 0 }  ][fill={rgb, 255:red, 0; green, 0; blue, 0 }  ][line width=0.75]      (0, 0) circle [x radius= 3.35, y radius= 3.35]   ;
\draw [shift={(275.5,66.5)}, rotate = 303.16] [color={rgb, 255:red, 0; green, 0; blue, 0 }  ][fill={rgb, 255:red, 0; green, 0; blue, 0 }  ][line width=0.75]      (0, 0) circle [x radius= 3.35, y radius= 3.35]   ;
\draw [shift={(251,104)}, rotate = 303.16] [color={rgb, 255:red, 0; green, 0; blue, 0 }  ][fill={rgb, 255:red, 0; green, 0; blue, 0 }  ][line width=0.75]      (0, 0) circle [x radius= 3.35, y radius= 3.35]   ;
\draw    (139,32) -- (193,109) ;
\draw [shift={(193,109)}, rotate = 54.96] [color={rgb, 255:red, 0; green, 0; blue, 0 }  ][fill={rgb, 255:red, 0; green, 0; blue, 0 }  ][line width=0.75]      (0, 0) circle [x radius= 3.35, y radius= 3.35]   ;
\draw [shift={(166,70.5)}, rotate = 54.96] [color={rgb, 255:red, 0; green, 0; blue, 0 }  ][fill={rgb, 255:red, 0; green, 0; blue, 0 }  ][line width=0.75]      (0, 0) circle [x radius= 3.35, y radius= 3.35]   ;
\draw [shift={(139,32)}, rotate = 54.96] [color={rgb, 255:red, 0; green, 0; blue, 0 }  ][fill={rgb, 255:red, 0; green, 0; blue, 0 }  ][line width=0.75]      (0, 0) circle [x radius= 3.35, y radius= 3.35]   ;
\draw    (179,149) -- (222,149) ;
\draw    (126,149) -- (179,149) ;
\draw [shift={(179,149)}, rotate = 0] [color={rgb, 255:red, 0; green, 0; blue, 0 }  ][fill={rgb, 255:red, 0; green, 0; blue, 0 }  ][line width=0.75]      (0, 0) circle [x radius= 3.35, y radius= 3.35]   ;
\draw [shift={(126,149)}, rotate = 0] [color={rgb, 255:red, 0; green, 0; blue, 0 }  ][fill={rgb, 255:red, 0; green, 0; blue, 0 }  ][line width=0.75]      (0, 0) circle [x radius= 3.35, y radius= 3.35]   ;
\draw    (221,104) -- (222,149) ;
\draw    (222,149) -- (222,231) ;
\draw [shift={(222,231)}, rotate = 90] [color={rgb, 255:red, 0; green, 0; blue, 0 }  ][fill={rgb, 255:red, 0; green, 0; blue, 0 }  ][line width=0.75]      (0, 0) circle [x radius= 3.35, y radius= 3.35]   ;
\draw [shift={(222,190)}, rotate = 90] [color={rgb, 255:red, 0; green, 0; blue, 0 }  ][fill={rgb, 255:red, 0; green, 0; blue, 0 }  ][line width=0.75]      (0, 0) circle [x radius= 3.35, y radius= 3.35]   ;
\draw [shift={(222,149)}, rotate = 90] [color={rgb, 255:red, 0; green, 0; blue, 0 }  ][fill={rgb, 255:red, 0; green, 0; blue, 0 }  ][line width=0.75]      (0, 0) circle [x radius= 3.35, y radius= 3.35]   ;
\draw    (219,23) -- (221,104) ;
\draw [shift={(221,104)}, rotate = 88.59] [color={rgb, 255:red, 0; green, 0; blue, 0 }  ][fill={rgb, 255:red, 0; green, 0; blue, 0 }  ][line width=0.75]      (0, 0) circle [x radius= 3.35, y radius= 3.35]   ;
\draw [shift={(220,63.5)}, rotate = 88.59] [color={rgb, 255:red, 0; green, 0; blue, 0 }  ][fill={rgb, 255:red, 0; green, 0; blue, 0 }  ][line width=0.75]      (0, 0) circle [x radius= 3.35, y radius= 3.35]   ;
\draw [shift={(219,23)}, rotate = 88.59] [color={rgb, 255:red, 0; green, 0; blue, 0 }  ][fill={rgb, 255:red, 0; green, 0; blue, 0 }  ][line width=0.75]      (0, 0) circle [x radius= 3.35, y radius= 3.35]   ;

\draw (146,24.4) node [anchor=north west][inner sep=0.75pt]    {$-3$};
\draw (166,54.4) node [anchor=north west][inner sep=0.75pt]    {$-2$};
\draw (192,93.4) node [anchor=north west][inner sep=0.75pt]    {$-2$};
\draw (301,24.4) node [anchor=north west][inner sep=0.75pt]    {$-3$};
\draw (262,92.4) node [anchor=north west][inner sep=0.75pt]    {$-2$};
\draw (282,62.4) node [anchor=north west][inner sep=0.75pt]    {$-2$};
\draw (247,171.4) node [anchor=north west][inner sep=0.75pt]    {$-4$};
\draw (269,203.4) node [anchor=north west][inner sep=0.75pt]    {$-5$};
\draw (198,173.4) node [anchor=north west][inner sep=0.75pt]    {$-4$};
\draw (179,203.4) node [anchor=north west][inner sep=0.75pt]    {$-5$};
\draw (113,126.4) node [anchor=north west][inner sep=0.75pt]    {$-3$};
\draw (165,126.4) node [anchor=north west][inner sep=0.75pt]    {$-2$};
\draw (268,127.4) node [anchor=north west][inner sep=0.75pt]    {$-3$};
\draw (300,127.4) node [anchor=north west][inner sep=0.75pt]    {$-6$};
\draw (231,128.4) node [anchor=north west][inner sep=0.75pt]    {$-4$};
\draw (227,24.4) node [anchor=north west][inner sep=0.75pt]    {$-3$};
\draw (225,54.4) node [anchor=north west][inner sep=0.75pt]    {$-2$};
\draw (225,93.4) node [anchor=north west][inner sep=0.75pt]    {$-2$};
\draw (223,174.4) node [anchor=north west][inner sep=0.75pt]    {$-4$};
\draw (226,203.4) node [anchor=north west][inner sep=0.75pt]    {$-5$};

\end{tikzpicture}

\caption{}
    \label{fig:ex1}
\end{figure}
\end{example}

\begin{remark}
Note that with Theorem \ref{thm:2} we have actually proved that, in the case of quotient representations of PM semigroups with convenient multiplicities, the first inclusion of (\ref{eq:sumbound}) is in fact an equality. Based on this, the question arises whether this is possible more generally or what are the limits of the argument given in the proof of Theorem \ref{thm:2}?

Unfortunately, it turns out that argument relies essentially on the two-legged form of the representations and for the general case it can not be applied.  

Indeed, consider for example the graphs $\Gamma_1$ and $\Gamma_2$ defined by $Sf_{\Gamma_1}=(-1;(15,4),5\times(15,2))$ and 
$Sf_{\Gamma_2}=(-1;(30,1),3\times(10,3))$, respectively. (Here $t\times(\al,\w)$ means there are $t$ number of legs characterized by the same pair $(\al,\w)$.) Suppose that there exists $a,b\in\Z_{\geq0}$
such that for $\Gamma:=a\Gamma_1+b\Gamma_2$ we have $\cS_{\Gamma}=\cS_{\Gamma_1}\cap\cS_{\Gamma_2}$. In this case, for the gaps of $\cS_{\Gamma}$ $\ell=15,20$ we get the following values

\begin{center}
\begin{tabular}{c|c|c}
     $\ell$&15&20  \\
     \hline
     $N_{\Gamma_1}(\ell)$&1&-1\\
     \hline
     $N_{\Gamma_2}(\ell)$&-1&1\\
     \hline
     $N(\ell)$&$a-b$&$b-a$
\end{tabular},  
\end{center}
Thus $a-b<0$ and $b-a<0$, a contradiction.

The reason is that, when $d_1=d_2=2$, 
either $|e_1|\leq|e_2|$ or $|e_2|\leq|e_1|$; after possibly swapping the indices, one may assume the first inequality.
In the general case, for similar methods to work we would 
need either $|e_1|(d_2-1)\leq|e_2|$ or $|e_2|(d_1-1)\leq |e_1|$. However, the example above shows that these inequalities can fail simultaneously.

\end{remark}

\subsection{Representations with Toms decomposition}

Let $\cS$ be an SPM numerical semigroup and now let  $\cS=\cP_1\cap\dots\cap\cP_r$ be a Toms decomposition. That is, the PM semigroups are written as  $\cP_i=\langle \al_{i1},\al_{i2}\rangle/n$ for any $i\in\{1,\dots,r\}$, $\al_{i1},\al_{i2},n$ satisfying the properties of section \ref{ss:Toms}.

In this case we have two ways to construct the representations. Indeed, either we construct the quotient representations $\Gamma_i$ associated with $\cP_i$ for every $i$ and then consider their sum with the convenient multiplicities as in section \ref{s:gamma}.

Or, we can first construct the sum  of the graphs $\Gamma_{\langle  \al_{i1}, \al_{i2}\rangle}$ with the   `same' multiplicities $q_i$ (defined analogously as in (\ref{eq:mult}) with $k_{ij}:=\lceil \al_{i1}\al_{i2}/\al_{j1}\al_{j1} \rceil$), and then take the quotient by $n$. In fact, one can see from sections \ref{s:PMquotrep} and \ref{s:gamma} that in both ways we arrive to the same representation $\Gamma$ of $\cS$. 

\begin{example}
\label{ex:405}

Consider again the semigroup $\cS = \langle 11, 14, 15, 18, 19, 21\rangle$ from Example \ref{ex:toms}, but at this time  with the Toms decomposition $\cS=\cP_1 \cap \cP_2$, where 
$\cP_1 = \langle 7, 11, 15, 19\rangle =\langle 7,19\rangle/3$ and $\cP_2 = \langle 5, 8, 11, 14 \rangle=\langle 5,14\rangle/3$. The corresponding quotient representations are $\Gamma_1$ defined by $Sf_1 = (-1, (7, 5), (19, 5))$ and $\Gamma_2$ defined by $Sf_2 = (-1, (5, 3), (14, 5))$. One calculates that $k_{12}=\lceil e_2/e_1\rceil-1=\lceil 133/70\rceil-1=1$, hence the multiplicities are $q_1=2$ and $q=1$. Therefore, we will get $\Gamma=2\Gamma_1+\Gamma_2$ as a representation for $\cS$, whose Seifert invariants are $(-3; (5,3), 2\times(7,5),(14,5), 2\times(19,5))$.

On the other hand, it can be verified, that choosing the other way of construction given by the advantage of the Toms decomposition, we would arrive first to a graph with Seifert invariants $(-3; (5,1), 2\times(7,4),(14,11), 2\times (19,8))$. Then, applying to this the quotient construction by $3$ we arrive to the same graph as in the previous way, see Figure \ref{fig:2}. 
\begin{figure}[!h]
\centering
\tikzset{every picture/.style={line width=0.75pt}} 

\begin{tikzpicture}[x=0.75pt,y=0.75pt,yscale=-1,xscale=1]

\draw    (213,129) -- (271,209) ;
\draw [shift={(271,209)}, rotate = 54.06] [color={rgb, 255:red, 0; green, 0; blue, 0 }  ][fill={rgb, 255:red, 0; green, 0; blue, 0 }  ][line width=0.75]      (0, 0) circle [x radius= 3.35, y radius= 3.35]   ;
\draw [shift={(242,169)}, rotate = 54.06] [color={rgb, 255:red, 0; green, 0; blue, 0 }  ][fill={rgb, 255:red, 0; green, 0; blue, 0 }  ][line width=0.75]      (0, 0) circle [x radius= 3.35, y radius= 3.35]   ;
\draw [shift={(213,129)}, rotate = 54.06] [color={rgb, 255:red, 0; green, 0; blue, 0 }  ][fill={rgb, 255:red, 0; green, 0; blue, 0 }  ][line width=0.75]      (0, 0) circle [x radius= 3.35, y radius= 3.35]   ;
\draw    (271,209) -- (296,246) ;
\draw [shift={(296,246)}, rotate = 55.95] [color={rgb, 255:red, 0; green, 0; blue, 0 }  ][fill={rgb, 255:red, 0; green, 0; blue, 0 }  ][line width=0.75]      (0, 0) circle [x radius= 3.35, y radius= 3.35]   ;
\draw    (242,169) -- (334,170) ;
\draw [shift={(334,170)}, rotate = 0.62] [color={rgb, 255:red, 0; green, 0; blue, 0 }  ][fill={rgb, 255:red, 0; green, 0; blue, 0 }  ][line width=0.75]      (0, 0) circle [x radius= 3.35, y radius= 3.35]   ;
\draw [shift={(288,169.5)}, rotate = 0.62] [color={rgb, 255:red, 0; green, 0; blue, 0 }  ][fill={rgb, 255:red, 0; green, 0; blue, 0 }  ][line width=0.75]      (0, 0) circle [x radius= 3.35, y radius= 3.35]   ;
\draw    (242,169) -- (186,249) ;
\draw [shift={(186,249)}, rotate = 124.99] [color={rgb, 255:red, 0; green, 0; blue, 0 }  ][fill={rgb, 255:red, 0; green, 0; blue, 0 }  ][line width=0.75]      (0, 0) circle [x radius= 3.35, y radius= 3.35]   ;
\draw [shift={(214,209)}, rotate = 124.99] [color={rgb, 255:red, 0; green, 0; blue, 0 }  ][fill={rgb, 255:red, 0; green, 0; blue, 0 }  ][line width=0.75]      (0, 0) circle [x radius= 3.35, y radius= 3.35]   ;
\draw [shift={(242,169)}, rotate = 124.99] [color={rgb, 255:red, 0; green, 0; blue, 0 }  ][fill={rgb, 255:red, 0; green, 0; blue, 0 }  ][line width=0.75]      (0, 0) circle [x radius= 3.35, y radius= 3.35]   ;
\draw    (242,169) -- (271,124) ;
\draw    (271,124) -- (320,49) ;
\draw [shift={(320,49)}, rotate = 303.16] [color={rgb, 255:red, 0; green, 0; blue, 0 }  ][fill={rgb, 255:red, 0; green, 0; blue, 0 }  ][line width=0.75]      (0, 0) circle [x radius= 3.35, y radius= 3.35]   ;
\draw [shift={(295.5,86.5)}, rotate = 303.16] [color={rgb, 255:red, 0; green, 0; blue, 0 }  ][fill={rgb, 255:red, 0; green, 0; blue, 0 }  ][line width=0.75]      (0, 0) circle [x radius= 3.35, y radius= 3.35]   ;
\draw [shift={(271,124)}, rotate = 303.16] [color={rgb, 255:red, 0; green, 0; blue, 0 }  ][fill={rgb, 255:red, 0; green, 0; blue, 0 }  ][line width=0.75]      (0, 0) circle [x radius= 3.35, y radius= 3.35]   ;
\draw    (159,52) -- (213,129) ;
\draw [shift={(213,129)}, rotate = 54.96] [color={rgb, 255:red, 0; green, 0; blue, 0 }  ][fill={rgb, 255:red, 0; green, 0; blue, 0 }  ][line width=0.75]      (0, 0) circle [x radius= 3.35, y radius= 3.35]   ;
\draw [shift={(186,90.5)}, rotate = 54.96] [color={rgb, 255:red, 0; green, 0; blue, 0 }  ][fill={rgb, 255:red, 0; green, 0; blue, 0 }  ][line width=0.75]      (0, 0) circle [x radius= 3.35, y radius= 3.35]   ;
\draw [shift={(159,52)}, rotate = 54.96] [color={rgb, 255:red, 0; green, 0; blue, 0 }  ][fill={rgb, 255:red, 0; green, 0; blue, 0 }  ][line width=0.75]      (0, 0) circle [x radius= 3.35, y radius= 3.35]   ;
\draw    (164,169) -- (242,169) ;
\draw [shift={(242,169)}, rotate = 0] [color={rgb, 255:red, 0; green, 0; blue, 0 }  ][fill={rgb, 255:red, 0; green, 0; blue, 0 }  ][line width=0.75]      (0, 0) circle [x radius= 3.35, y radius= 3.35]   ;
\draw [shift={(203,169)}, rotate = 0] [color={rgb, 255:red, 0; green, 0; blue, 0 }  ][fill={rgb, 255:red, 0; green, 0; blue, 0 }  ][line width=0.75]      (0, 0) circle [x radius= 3.35, y radius= 3.35]   ;
\draw [shift={(164,169)}, rotate = 0] [color={rgb, 255:red, 0; green, 0; blue, 0 }  ][fill={rgb, 255:red, 0; green, 0; blue, 0 }  ][line width=0.75]      (0, 0) circle [x radius= 3.35, y radius= 3.35]   ;

\draw (166,44.4) node [anchor=north west][inner sep=0.75pt]    {$-3$};
\draw (186,74.4) node [anchor=north west][inner sep=0.75pt]    {$-2$};
\draw (212,113.4) node [anchor=north west][inner sep=0.75pt]    {$-2$};
\draw (321,44.4) node [anchor=north west][inner sep=0.75pt]    {$-3$};
\draw (282,112.4) node [anchor=north west][inner sep=0.75pt]    {$-2$};
\draw (302,82.4) node [anchor=north west][inner sep=0.75pt]    {$-2$};
\draw (267,191.4) node [anchor=north west][inner sep=0.75pt]    {$-4$};
\draw (289,223.4) node [anchor=north west][inner sep=0.75pt]    {$-5$};
\draw (221,193.4) node [anchor=north west][inner sep=0.75pt]    {$-4$};
\draw (195,236.4) node [anchor=north west][inner sep=0.75pt]    {$-5$};
\draw (151,146.4) node [anchor=north west][inner sep=0.75pt]    {$-3$};
\draw (189,145.4) node [anchor=north west][inner sep=0.75pt]    {$-2$};
\draw (288,147.4) node [anchor=north west][inner sep=0.75pt]    {$-3$};
\draw (320,147.4) node [anchor=north west][inner sep=0.75pt]    {$-5$};
\draw (233,133.4) node [anchor=north west][inner sep=0.75pt]    {$-3$};
\end{tikzpicture}
\caption{}\label{fig:2}
\end{figure}
\end{example}

\begin{remark}
We start this remark with the following observation: in terms of the previous Example \ref{ex:405} the graph $\Gamma=\Gamma_1+\Gamma_2$  is also a representation for the SPM semigroup considered in previous examples.  

The point is that, in many concrete examples we can relax on the multiplicities and they can be smaller in the construction of a representation. However, this is not a general behavior, as will be justified by the next example. 
\end{remark}

\begin{example}
Consider the SPM semigroup $\cS=\langle 41,42,60,100,105,140\rangle$ and its decomposition 
$\cS=\cP_1\cap \cP_2$, where $\cP_1=\langle   20, 21\rangle=\langle 20,21\rangle/1$ and $\cP_2=\langle  6,35 \rangle=\langle 6,35\rangle/1$. 

The quotient representations $\Gamma_i$ associated with $\cP_i$ are as follows: $\Gamma_1$ is defined by the Seifert invariants $Sf_{\Gamma_1}=(-1;(20,19),(21,1))$, while $\Gamma_2$ is defined by $Sf_{\Gamma_2}=(-1;(35,29),(6,1))$. A direct calculation gives $|e_1|=1/420$, $|e_2|=1/210$ and $k_{12}=\lceil e_2/e_1\rceil-1=1$. Thus the multiplicities are
$$q_1=k_{12}+1=2 \quad \mbox{and} \quad q_2=1$$ 
which, by (\ref{SPMrep}), provides the graph  $\Gamma=2\Gamma_1+\Gamma_2$. In this case, however, $q_1\geq 2$ is necessary since $N_{\Gamma_1+\Gamma_2}(216)\geq0$, but $216\notin\cP_1\cap\cP_2$.

\end{example}

\subsection{Further remarks and questions}
We have just proved that every SPM numerical semigroup admits representations by canonical equivariant graphs defined by Seifert invariants of general type $Sf=(-d;(\al_{1},\omega_1),\dots, (\al_{2d},\omega_{2d}))$, i.e. they have $2d$ legs and $-d$ decoration on the central vertex.

Then we can ask the question whether the numerical semigroups associated with these type of graphs are always SPM semigroups? Unfortunately, the answer is no, as will be shown by the following example.

\begin{example}
Consider the graph defined by the Seifert invariants $(-2; (2,1), (3,2), (3,2), \\(7,1))$. If we compute the generators we obtain that $S_{\Gamma} = \langle 6, 21, 28 \rangle$. Moreover, by the algorithm of \cite{delgado2008} we can check that in fact it is not an SPM.
\end{example}
Motivated by this example, we end this section by posing the following questions:

\begin{itemize}
\item Is there any extra condition on the above type of graphs which would characterize those among them, whose associated semigroups are SPM?

\item How can the generators and special formulas for invariants be understood and deduced by using the underlying geometry of these representations?
\end{itemize}

\bibliographystyle{amsplain}

\end{document}